\documentclass{amsart}

\usepackage{enumerate}
\usepackage{amsmath}%
\usepackage{amsfonts}%
\usepackage{amssymb}%
\usepackage{graphicx}
\usepackage{mathrsfs}
\usepackage{hyperref}
\usepackage{fullpage}
\newtheorem{theorem}{Theorem}
\theoremstyle{plain}

\newtheorem{conjecture}[theorem]{Conjecture}
\newtheorem{corollary}[theorem]{Corollary}

\newtheorem{fact}[theorem]{Fact}

\newtheorem{proposition}[theorem]{Proposition}

\numberwithin{equation}{section}
\numberwithin{theorem}{section}
\numberwithin{case}{section}

\numberwithin{subcase}{case}
\begin{document}
\title[Forbidding Hamilton cycles in hypergraphs]{Forbidding Hamilton cycles in uniform hypergraphs}
\thanks{
The first author is supported by FAPESP (Proc. 2014/18641-5).
The second author is partially supported by NSF grant DMS-1400073.}
\author{Jie Han}
\address{Instituto de Matem\'{a}tica e Estat\'{\i}stica, Universidade de S\~{a}o Paulo, Rua do Mat\~{a}o 1010, 05508-090, S\~{a}o Paulo, Brazil}
\email[Jie Han]{jhan@ime.usp.br}
\author{Yi Zhao}
\address
{Department of Mathematics and Statistics, Georgia State University, Atlanta, GA 30303, USA}
\email[Yi Zhao]{yzhao6@gsu.edu}

\date{\today}

\begin{abstract}
For $1\le d\le \ell< k$, we give a new lower bound for the minimum $d$-degree threshold that guarantees a Hamilton $\ell$-cycle in $k$-uniform hypergraphs. When $k\ge 4$ and $d< \ell=k-1$, this bound is larger than the conjectured minimum $d$-degree threshold for perfect matchings and thus disproves a well-known conjecture of R\"odl and Ruci\'nski. Our (simple) construction generalizes a construction of Katona and Kierstead and the space barrier for Hamilton cycles.
\end{abstract}

\maketitle

\section{Introduction}

The study of Hamilton cycles is an important topic in graph theory. A classical result of Dirac \cite{Dirac} states that every graph on $n\ge 3$ vertices with minimum degree $n/2$ contains a Hamilton cycle. In recent years, researchers have worked on extending this theorem to hypergraphs -- see recent surveys \cite{KuOs14ICM, RR, Zhao15}.

To define Hamilton cycles in hypergraphs, we need the following definitions. Given $k\ge 2$, a $k$-uniform hypergraph (in short, \emph{$k$-graph}) consists of a vertex set $V$ and an edge set $E\subseteq \binom{V}{k}$, where every edge is a $k$-element subset of $V$.
Given a $k$-graph $H$ with a set $S$ of $d$ vertices (where $1 \le d \le k-1$) we define $\deg_{H} (S)$ to be the number of edges containing $S$ (the subscript $H$ is omitted if it is clear from the context). The \emph{minimum $d$-degree $\delta _{d} (H)$} of $H$ is the minimum of $\deg_{H} (S)$ over all $d$-vertex sets $S$ in $H$.
%We refer to  $\delta _1 (H)$ as the \emph{minimum vertex degree} and  $\delta _{k-1} (H)$ the \emph{minimum codegree} of $H$.
For $1\le \ell \le k-1$, a $k$-graph is a called an \emph{$\ell$-cycle} if its vertices can be ordered cyclically such that each of its edges consists of $k$ consecutive vertices and every two consecutive edges (in the natural order of the edges) share exactly $\ell$ vertices. In $k$-graphs, a $(k-1)$-cycle is often called a \emph{tight} cycle. We say that a $k$-graph contains a \emph{Hamilton $\ell$-cycle} if it contains an $\ell$-cycle as a spanning subhypergraph. Note that a Hamilton $\ell$-cycle of a $k$-graph on $n$ vertices contains exactly $n/(k - \ell)$ edges, implying that $k- \ell$ divides $n$.

Let $1\le d, \ell \le k-1$. For $n\in (k - \ell)\mathbb{N}$, we define $h_d^\ell(k,n)$ to be the smallest integer $h$ such that every $n$-vertex $k$-graph $H$ satisfying $\delta_d(H)\ge h$ contains a Hamilton $\ell$-cycle.
Note that whenever we write $h_d^{\ell}(k,n)$, we always assume that $1\le d\le k-1$.
Moreover, we often write $h_d(k,n)$ instead of $h_d^{k-1}(k,n)$ for simplicity.
Similarly, for $n\in k\mathbb{N}$, we define {${m_d(k,n)}$} to be the smallest integer $m$ such that every $n$-vertex $k$-graph $H$ satisfying $\delta_d(H)\ge m$ contains a perfect matching.
The problem of determining $m_d(k,n)$ has attracted much attention recently and the asymptotic value of $m_d(k,n)$ is conjectured as follows. Note that the $o(1)$ term refers to a function that tends to $0$ as $n\to \infty$ throughout the paper.

\begin{conjecture}\cite{HPS, KuOs-survey}
\label{conj:mat}
For $1\le d\le k-1$ and $k \mid n$,
\[
m_d (k,n) = \left(\max \left\{ \frac12, 1- \left(1- \frac{1}{k} \right)^{k-d}  \right\} +o(1)\right)\binom{n-d}{k-d}.
\]
\end{conjecture}
Conjecture~\ref{conj:mat} has been confirmed  \cite{AFHRRS, Pik} for $\min\{k-4, k/2\}\le d\le k-1$ (the exact values of $m_d(k, n)$ are also known in some cases, e.g., \cite{RRS09, TrZh13}). On the other hand, 
$h_d^\ell(k,n)$ has also been extensively studied \cite{BHS, CzMo, GPW, HS, HZk2, HZ2, HZ1, KK, KKMO, KMO, KO, RoRu14, RRS06, RRS08, RRS11}.
In particular, R\"odl, Ruci\'nski and Szemer\'edi \cite{RRS06, RRS08} showed that $h_{k-1}(k,n) = (1/2+o(1))n$.
The same authors proved in \cite{RRS06mat} that $m_{k-1}(k,n) = (1/2+o(1))n$ (later they determined $m_{k-1}(k,n)$ exactly \cite{RRS09}).
This suggests that the values of $h_d(k,n)$ and $m_d(k,n)$ are closely related and inspires R\"{o}dl and Ruci\'nski to make the following conjecture.

%In \cite{RR}, R\"{o}dl and Ruci\'nski posed the following conjecture on $h_d^{k-1}(k,n)$ (see also \cite[Conjecture 5.3]{KuOs14ICM} K\"uhn and Osthus).

\begin{conjecture}\cite[Conjecture 2.18]{RR}\label{conj:1}
Let $k\ge 3$ and $1\le d\le k-2$. Then
\[
h_d(k,n) = m_d(k,n) + o(n^{k-d}).
%h_d(k,n) \approx m_d(k,n).
\]
\end{conjecture}

By using the value of $m_d(k, n)$ from Conjecture~\ref{conj:mat}, K\"uhn and Osthus stated this conjecture explicitly for the case $d=1$.  
\begin{conjecture}\cite[Conjecture 5.3]{KuOs14ICM}\label{conj:2}
Let $k\ge 3$. Then
\[
h_1(k,n) = \left( 1- \left(1- \frac{1}{k} \right)^{k-1} + o(1)\right)\binom{n-1}{k-1}.
\]
\end{conjecture}

%By the result of \cite{RRS09}, Conjecture~\ref{conj:1} is true for $d'=k-1$

In this note we provide new lower bounds for $h_d^{\ell}(k,n)$ when $d\le \ell$.
%In particular, for $\ell=k-1$, this disproves Conjecture~\ref{conj:1} for all the cases when the values of $m_d(k,n)$ are known except $k=3$ and $d=k-1$, and Conjecture~\ref{conj:2}.

\begin{theorem}\label{prop:cons}
Let $1\le d \le k-1$ and $t=k-d$, then
\[
h_d(k,n) \ge  \Bigg( 1 - \binom t{ \lfloor t/2 \rfloor} \frac{ \lceil t/2\rceil^{\lceil t/2\rceil} (\lfloor t/2 \rfloor+1) ^{\lfloor t/2 \rfloor}}{(t+1)^t} +o(1) \Bigg) \binom nt.
\]
\end{theorem}

\begin{theorem}\label{prop:cons2}
Let $1\le d \le \ell \le k-1$ and $t=k-d$. Then
\[
h_d^{\ell}(k,n) \ge \left( 1 - b_{t, k-\ell} {2^{-t}} + o(1) \right)\binom nt ,
\]
where $b_{t, k-\ell}$ equals the largest sum of the $k-\ell$ consecutive binomial coefficients from $\binom{t}{0},\dots, \binom tt$.
\end{theorem}

Theorem~\ref{prop:cons} disproves both Conjectures~\ref{conj:1} and \ref{conj:2}.

\begin{corollary}\label{cor:disprove}
%that $h_{k-2}(k,n) \ge (\frac 59+o(1)) \binom n2$, $h_{k-3}(k,n) \ge (\frac 58+o(1)) \binom n2$ and $h_{k-4}(k,n) \ge (\frac{409}{625}+o(1)) \binom n2$.
For all $k$,
\begin{align*}
h_{k-2}(k,n) \ge \left(\frac 59+o(1)\right) \binom n2, \ h_{k-3}(k,n) \ge \left(\frac 58+o(1)\right) \binom n3,  \ h_{k-4}(k,n) \ge \left(\frac{409}{625}+o(1)\right) \binom n4
\end{align*}
and in general, for any $1\le d \le k-1$,
\begin{align}
h_d (k,n) & %\ge \Bigg(1 - \binom{t}{\lfloor t/2 \rfloor}\frac1{2^t} \Bigg) \binom{n}{t}
> \left(1 - \frac1{\sqrt{3(k-d)/2+1}} \right) \binom{n}{k-d}. \label{eq:Cor1}
%& \ge  \left(\max \left\{ \frac12, \, 1- \left(1- \frac{1}{k} \right)^{k-d}  \right\} +o(1)\right)\binom{n}{k-d}. \label{eq:Cor2}
\end{align}
These bounds imply that Conjecture~\ref{conj:1} is false when $k\ge 4$ and $\min\{k-4, k/2\}\le d\le k-2$, and Conjecture~\ref{conj:2} is false whenever $k\ge 4$.
\end{corollary}

We will prove Theorem~\ref{prop:cons}, Theorem~\ref{prop:cons2}, and Corollary~\ref{cor:disprove} in the next section.

\medskip
We believe that Conjecture~\ref{conj:1} is false whenever $k\ge 4$ but due to our limited knowledge on $m_d(k,n)$, we can only disprove Conjecture~\ref{conj:1} for the cases when $m_d(k,n)$ is known.

%
%
%
%\begin{theorem}\cite{AFHRRS, KOTo}\label{thm:mat}
%Conjecture~\ref{conj:mat} holds for $k\ge 3$ and positive integer $d$ such that $\min\{k-4, k/2\}\le d\le k-2$.
%\end{theorem}

%Theorem~\ref{prop:cons} shows that  for all $k\ge 3$.
This bound $h_{k-2}(k, n)\ge ( \frac59 + o(1))\binom n2$ coincides with the value of $m_1(3, n)$ --  it was shown in \cite{HPS} that $m_1(3, n)= (5/9 +o(1)) \binom n2$, and it was widely believed that $h_1(3, n)= (5/9 + o(1))\binom n2$, e.g., see \cite{RoRu14}.
On the other hand, it is known  \cite{Pik} that $m_2(4, n)= (\frac12+o(1)) \binom n2$, which is smaller than $\frac59 \binom n2$. Therefore $k=4$ and $d=2$ is the smallest case when Theorem~\ref{prop:cons} disproves Conjecture~\ref{conj:1}.
%In addition, Theorem~\ref{prop:cons} shows that $h^{k-1}_{k-3}(k, n)\ge ( \frac58 + o(1))\binom n3$ for all $k\ge 4$ and $h^{k-1}_{k-4}(k, n)\ge ( \frac{409}{625} + o(1))\binom n4$ for all $k\ge 5$.
More importantly, \eqref{eq:Cor1} shows that $h_d(k, n)/\binom{n}{k-d}$ tends to one as $k-d$ tends to $\infty$.  For example, as $k$ becomes sufficiently large, $h_{k - \ln k}(k,n)$ is close to $\binom{n-d}{k-d}$, the trivial upper bound. In contrast, Conjecture~\ref{conj:mat} suggests that there exists $c>0$ independent of $k$ and $d$ ($c= 1/e$, where $e=2.718...$, if Conjecture~\ref{conj:mat} is true) such that $m_d(k, n)\le (1-c)\binom{n-d}{k-d}$.
%Note that $1- 1/\sqrt{3t/2+1}$ tends to 1 when $t$ tends to infinity.

Similarly, by Theorem~\ref{prop:cons2}, if $k-\ell = o(\sqrt{t})$, %as $t = k-d$ tends to $\infty$, then  implies
$h_d^{\ell}(k,n)/ \binom{n}{t}$ tends to one as $t$ tends to $\infty$ because
\[
1 - b_{t, k-\ell} {2^{-t}} \ge 1- \frac{k-\ell}{2^t}\binom {t}{ \lfloor t/2 \rfloor} \approx 1- \frac{o(\sqrt t)}{\sqrt{\pi t/2}}.
\]
Theorem~\ref{prop:cons2} also implies the following special case: suppose $k$ is odd  and $\ell=d=k-2$.
Then $t=2$ and $b_{t, k-\ell}=b_{2,2} = 3$, and consequently $h_{k-2}^{k-2}(k,n) \ge \left( \frac14 + o(1) \right)\binom n2 $. Previously it was only known that $h_{k-2}^{k-2}(k,n)\ge (1 - (\frac{k}{k+1})^2+o(1))\binom n2$ from
\eqref{eq:sb} (where $a=\lceil k/(k-\ell) \rceil= (k+1)/2$). When $k$ is large, the bound provided by Theorem~\ref{prop:cons2} is much better.
%\frac2{k+1} - \frac1{(k+1)^2} $k-\ell \nmid k$ as $k$ is odd
%Thus, when $k\ge 7$, Theorem~\ref{prop:cons2} gives a new lower bound for $h_{k-2}^{k-2}(k,n)$.

Finally, we do not know if Theorems~\ref{prop:cons} and \ref{prop:cons2} are best possible. Glebov, Person, and Weps \cite{GPW} gave a general upper bound (far away from our lower bounds)
\[
h_d^{\ell}(k, n)\le \left(1 -  \frac{1}{c k^{3k-3}} \right) \binom{n-d}{k-d},
\]
where $c$ is a constant independent of $d, \ell, k, n$.
%Since determining $h_d^{\ell}(k, n)$ in general seems very hard,
%it is interesting to improve this upper bound (or our lower bounds) to find $s= s(d, \ell)$ such that
%\[
%1 - \frac{c_1}{k^s} \le \frac{h_d^{\ell}(k, n)}{\binom{n-d}{k-d}} \le 1 -  \frac{c_2}{k^{s}}
%\]
%for some constants $c_1, c_2$ independent of $k, n$.

\section{The proofs}

Before proving our results, it is instructive to recall the so-called \emph{space barrier}.

\begin{proposition}\cite{KMO}\label{prop:sb}
Let $H=(V, E)$ be an $n$-vertex $k$-graph such that $V=X\dot\cup Y$
\footnote{Throughout the paper, we write $X\dot\cup Y$ for $X\cup Y$ when sets $X$, $Y$ are disjoint.}
and $E=\{e\in \binom Vk: e\cap X \neq \emptyset \}$.
Suppose $|X| < \frac{1}{a(k-\ell)}n$, where $a:=\lceil k/(k-\ell) \rceil$, then $H$ does not contain a Hamilton $\ell$-cycle.
\end{proposition}

A proof of Proposition~\ref{prop:sb} can be found in \cite[Proposition~2.2]{KMO} and is actually
included in our proof of Proposition~\ref{clm:NO} below. It is not hard to see that Proposition~\ref{prop:sb} shows that
\begin{equation}
\label{eq:sb}
h_d^{\ell}(k,n) \ge \left( 1- \left(1- \frac{1}{a(k-\ell)} \right)^{k-d}  + o(1)\right)\binom{n-d}{k-d}.
\end{equation}

Now we state our construction for Hamilton cycles -- it generalizes the one given by Katona and Kierstead \cite[Theorem 3]{KK} (where $j= \lfloor k/2 \rfloor$) and the space barrier (where $j= \ell + 1 - k$) simultaneously. The special case of $k=3, \ell=2, j=1$, and $|X|=n/3$ appears in \cite[Construction 2]{RoRu14}.

\begin{proposition}\label{clm:NO}
Given an integer $j$ such that $\ell+1-k \le j\le k$, let $H=(V, E)$ be an $n$-vertex $k$-graph such that $V=X\dot\cup Y$ and $E=\{e\in \binom Vk: |e\cap X| \notin \{j, j+1,\dots, j+k-\ell-1\}$.
Suppose $\frac{j-1}{a'(k-\ell)}n < |X| < \frac{j+k-\ell}{a(k-\ell)}n$, where $a':=\lfloor k/(k-\ell) \rfloor$ and $a:=\lceil k/(k-\ell) \rceil$, then $H$ does not contain a Hamilton $\ell$-cycle.
\end{proposition}

\begin{proof}
Suppose instead, that $H$ contains a Hamilton $\ell$-cycle $C$. Then all edges $e$ of $C$ satisfy $|e\cap X| \notin \{j, j+1,\dots, j+k-\ell-1\}$.
We claim that either all edges $e$ of $C$ satisfy $|e\cap X|\le j-1$ or all edges $e$ of $C$ satisfy $|e\cap X|\ge j+k-\ell$. Otherwise, there must be two consecutive edges $e_1, e_{2}$ in $C$ such that $|e_1 \cap X|\le j-1$ and $|e_{2}\cap X|\ge j+k-\ell$. However, since $|e_1\cap e_2|=\ell$, we have $||e_1\cap X| - |e_2\cap X||\le k-\ell$, a contradiction.

Observe that every vertex of $H$ is contained in either $a$ or $a'$ edges of $C$ and $C$ contains $\frac n{k-\ell}$ edges.
This implies that
\[
a' |X| \le \sum_{e\in C} |e\cap X| \le a |X|.
\]
On the other hand, we have $\sum_{e\in C} |e\cap X| < (j-1) \frac n{k-\ell}$ or $\sum_{e\in C} |e\cap X| > (j+k-\ell) \frac n{k-\ell}$.
In either case, we get a contradiction with the assumption $\frac{j-1}{a'(k-\ell)}n < |X| < \frac{j+k-\ell}{a(k-\ell)}n$.
%Finally, note that we get the space barrier by letting $j+k-\ell-1=0$.
\end{proof}

%\medskip
%\noindent\textbf{Remark 1.}
%Note that $H$ given by Proposition~\ref{clm:NO} satisfies $\delta_{d}(H) = 0$ for $d>\ell - \min\{j, 0\}$.
%Indeed, for any $d \ge \ell+1 - \min\{j, 0\}$, consider any $d$-set $S$ such that $|S\cap X|= \max\{j, 0\}$.
%Note that any $k$-set $T$ containing $S$ satisfies that
%\[
%\max\{j, 0\}\le |T\cap X| \le \max\{j, 0\} + k-d \le j+k - \ell - 1,
%\]
%because $\max\{j, 0\} - j = -\min\{j,0\}$.
%This means that $T\notin E(H)$ and thus $\deg_H(S) = 0$.
%So the construction gives a nontrivial lower bound only if $d \le \ell - \min\{j, 0\}$.

Note that by reducing the lower and upper bounds for $|X|$ by small constants, we can conclude that $H$ actually contains no \emph{Hamilton $\ell$-path}.

\medskip

To prove Theorems~\ref{prop:cons} and \ref{prop:cons2}, we apply Proposition~\ref{clm:NO} with appropriate $j$ and $|X|$.
%it remains to estimate the optimal values of the minimum $d$-degree of the $k$-graphs in .
%%%%%%%%%%%TIGHT CYCLES%%%%%%%%%%%%%%%%
We need the following fact.

\begin{fact}\label{clm:j}
Let $k,d,t,j$ be integers such that $1\le d\le k-1$ and $t=k-d$. If $\frac{j-1}{k} < \frac{\lceil t/2 \rceil}{t+1} < \frac{j+1}k$, then $ j - d\le \lceil t/2 \rceil \le j$.
\end{fact}

\begin{proof}
Since $\frac{j-1}{k} < \frac{\lceil t/2 \rceil}{t+1} < \frac{j+1}k$, we get
\[
\frac{k \lceil t/2 \rceil}{t+1} -1 < j < \frac{k\lceil t/2 \rceil}{t+1} +1.
\]
We need to show that $\lceil t/2 \rceil \le j \le \lceil t/2 \rceil + d$. First,
\[
j< \frac{k\lceil t/2 \rceil}{t+1} +1 = \lceil t/2 \rceil + \frac{(k-t-1)\lceil t/2 \rceil}{t+1} +1 \le \lceil t/2 \rceil + d,
\]
because $\lceil t/2 \rceil \le t+1$ and $k-t=d$.
Second, $j> \frac{k \lceil t/2 \rceil}{t+1} -1\ge \lceil t/2 \rceil - 1$ as $k\ge t+1$, so $j\ge \lceil t/2 \rceil$.
\end{proof}

In the proofs of Theorems~\ref{prop:cons} and \ref{prop:cons2}, we will consider binomial coefficients $\binom pq$ with $q<0$ -- in this case $\binom pq = 0$. We will conveniently write $|X|= xn$, where $0< x< 1$, instead of $|X|= \lfloor xn \rfloor$ -- this does not affect our calculations as $n$ is sufficiently large.

\begin{proof}[Proof of Theorem~\ref{prop:cons}]
Let $x = \lceil t/2 \rceil /(t+1)$. Since $\bigcup_{j=1}^{k-1} (\frac{j-1}{k}, \frac{j+1}k) = (0, 1)$ and $1/3\le \frac{\lceil t/2 \rceil}{t+1} \le 1/2$, there exists an integer $j\in [k-1]$ such that $\frac{j-1}{k} < \frac{\lceil t/2 \rceil}{t+1} < \frac{j+1}k$.
Let $H=(V, E)$ be an $n$-vertex $k$-graph such that $V=X\dot\cup Y$, $|X|=xn$ and $E=\{e\in \binom Vk: |e\cap X| \neq j\}$. Since $\frac{j-1}{k}n < |X| < \frac{j+1}{k}n$, $H$ contains no tight Hamilton cycle by Proposition~\ref{clm:NO}.

%Fix $1\le d\le k-1$ and let $t=k-d$.
Now let us compute $\delta_{d}(H)$. For $0\le i\le d$, let $S_i$ be any $d$-vertex subset of $V$ that contains exactly $i$ vertices in $X$. By the definition of $H$,
\[
\deg_{H}(S_i) = \binom{n-d}{t} - \binom{|X|-i}{j - i} \binom{|Y| - (d-i)}{t-j +i}.
\]
Note that this holds for $i> j$ or $i< j - t$ trivially.
%so that the equation is also true for $i> j$ or $i< j - t$ -- indeed, for such $i$, it is true that $\deg_{H}(S_i) = \binom{n-d}{t}$.
So we have
\begin{align}
\delta_{d}(H) &= \min_{ 0 \le i\le d}\left\{ \binom{n-d}{t} - \binom{|X|-i}{j - i} \binom{|Y| - (d-i)}{t-j+i} \right\} \nonumber \\
& = \binom{n}{t} - \max_{j - d \le i'\le  j }\left\{  \binom{|X|}{i'} \binom{|Y|}{t-i'} \right\} + o(n^{t}). \nonumber
\end{align}
%where in the second equation, we substitute the sum by $i' = j -i$.
Write $|X| = xn$ and $|Y| = yn$. When $0\le i' \le t$, we have
\[
\binom{|X|}{i'} \binom{|Y|}{t-i'} = \frac{(xn)^{i'} (yn)^{t-i'}}{i'! (t-i')!} + o(n^{t}) = \binom{t}{i'} x^{i'} y^{t-i'} \binom nt + o(n^{t}).
\]
When $i'<0$ or $i'>t$, we have $\binom{|X|}{i'} \binom{|Y|}{t-i'} = 0 =  \binom{t}{i'} x^{i'} y^{t-i'} \binom nt $.
In all cases, we have
\[
\delta_{d}(H) = \binom{n}{t} - \max_{j - d \le i'\le j}\left\{ \binom{t}{i'} x^{i'} y^{t-i'} \right\} \binom nt + o(n^{t}).
\]

%For $0\le i\le d$, let $a_i = \binom d{i} x^{i} y^{d-i}$.
%Note that our goal is to maximize $\delta_{d'}(H_x)$ over $0\le x\le 1$, i.e., it suffices to determine $\min_{0\le x\le 1} \max_{0\le i\le d} \left\{ a_i \right\}$.

Let $a_i := \binom t{i} x^{i} y^{t-i}$. Since $x = \lceil t/2 \rceil /(t+1)$ and $y=1-x$, it is easy to see that $\max_{0 \le i\le t} a_{i} = a_{\lceil t/2 \rceil}$ (e.g., by observing $\frac{a_i}{a_{i+1}} = \frac yx \cdot \frac{i+1}{t-i}$ for $0\le i< t$).
%For $0\le i\le t$,
%Note that
%\[
%\frac{a_i}{a_{i+1}} = \frac{y \binom ti}{ x \binom{t}{i+1}} = \frac yx \cdot \frac{i+1}{t-i},
%\]
%which implies that $a_i \ge a_{i+1}$ is equivalent to $y \ge \frac{t-i}{i+1}x$, which in turn, is equivalent to $x  \le \frac{i+1}{t+1}$, as $x+y=1$.
Moreover, by Fact~\ref{clm:j}, we have $j - d\le \lceil t/2 \rceil \le j $.
Together with $x = \lceil t/2 \rceil /(t+1)$, this implies that
\[
\max_{j - d \le i\le j} \left\{ a_{i} \right\} = a_{\lceil t/2 \rceil} = \binom t{ \lceil t/2 \rceil} \frac{ \lceil t/2\rceil^{\lceil t/2\rceil} (\lfloor t/2 \rfloor+1) ^{\lfloor t/2 \rfloor}}{(t+1)^t}
\]
and thus the proof is complete.
\end{proof}

%\noindent\textbf{Remark 2.}
%The proof of Theorem~\ref{prop:cons} says that the optimal values of $\delta_{d}(H)$ do not always happen at $x=1/2$, i.e., an (almost) balanced bipartition $(X, Y)$.
%However, $x = \lceil t/2 \rceil /(t+1)=1/2$ when $t=k-d$ is odd.
%For $t$ is even, the improvement has been observed when $k=3$ and $d=1$ in \cite[Construction 2]{RoRu14} (then $x=1/3$).
%

%%%%%%%%%%%%%%%% LOOSE CYCLES %%%%%%%%%%%%%%%

\medskip

Now we turn to the proof of Theorem~\ref{prop:cons2}, in which we assume that $|X|=n/2$, though a further improvement of the lower bound may be possible by considering other values of $|X|$.
%Also, as explained in Remark 1, we assume $d\le \ell$ (because we will take $j=\lfloor k/2 \rfloor>0$ in the proof).

\begin{proof}[Proof of Theorem~\ref{prop:cons2}]
The proof is similar to the one of Theorem~\ref{prop:cons}.
Let $H=(V, E)$ be an $n$-vertex $k$-graph such that $V=X\dot\cup Y$, $|X|=n/2$ and $E=\{e\in \binom Vk: |e\cap X| \notin \{\lceil \ell/2 \rceil, \dots, \lceil \ell/2 \rceil + k- \ell - 1\}\}$.
Note that
\begin{align*}
a'(k-\ell) &= \left \lfloor \frac{k}{k-\ell} \right\rfloor (k-\ell) \ge k - (k-\ell - 1) = \ell+1 > 2(\lceil \ell/2 \rceil  - 1), \text{ and} \\
a(k-\ell) &= \left \lceil \frac{k}{k-\ell} \right\rceil (k-\ell) \le k + (k-\ell-1) < 2(k - \lfloor \ell/2 \rfloor) = 2(\lceil \ell/2 \rceil+k - \ell).
\end{align*}
So we have
\[
\frac{\lceil \ell/2 \rceil - 1}{a'(k-\ell)}n < |X|=\frac n2 < \frac{\lceil \ell/2 \rceil+k - \ell}{a(k-\ell)}n.
\]
Thus, $H$ contains no Hamilton $\ell$-cycle by Proposition~\ref{clm:NO}.

Fix $1\le d\le k-1$ and let $t=k-d$. Now we compute $\delta_{d}(H)$. For $0\le i\le d$, let $S_i$ be any $d$-vertex subset of $V$ that contains exactly $i$ vertices in $X$.
It is easy to see that
\[
\deg_{H}(S_i) = \binom{n}{t} - \sum_{p=i'}^{i'+k-\ell-1} \binom{|X|}{p} \binom{|Y|}{t-p} + o(n^t),
\]
where $i' = \lceil \ell/2 \rceil -i$. %(note that $\binom ab=0$ if $b<0$ so the equation holds for all $0\le i\le d$).
Using $|X|=|Y|=n/2$ and the similar calculations in the proof of Theorem~\ref{prop:cons}, we get
\[
\deg_{H}(S_i) = \binom{n}{t} - \sum_{p=i'}^{i'+k-\ell-1} \binom{t}{p} \frac1{2^t}\binom nt + o(n^t).
\]
By the definition of $b_{t, k-\ell}$, we have %the definition of $H$ and
\begin{align}
\delta_{d}(H) &= \min_{0\le i \le d} \deg_H(S_i) \ge \binom{n}{t} - b_{t, k-\ell} {2^{-t}} \binom nt + o(n^{t}) \nonumber. \qedhere
\end{align}
\end{proof}

\medskip

%%%%%%%%%%%%%%%% COROLLARY%%%%%%%%%%

Corollary~\ref{cor:disprove} follows from Theorem~\ref{prop:cons} via simple calculations.

\begin{proof}[Proof of Corollary~\ref{cor:disprove}]
Let $t= k-d$ and
\[
f(t):= \binom t{ \lfloor t/2 \rfloor } \frac{ \lceil t/2\rceil^{\lceil t/2\rceil} (\lfloor t/2 \rfloor+1) ^{\lfloor t/2 \rfloor}}{(t+1)^t}.
\]
Theorem~\ref{prop:cons} states that $h_{k-t}(k, n)\ge (1 - f(t) + o(1)) \binom nt$ for any $1\le t\le k-1$.
Since
\begin{equation*}
\label{eq:f234}
f(2)= \frac49, \quad f(3)=\frac38, \quad \text{and} \quad f(4)= \frac{216}{625},
\end{equation*}
the bounds for $h_{k-t}(k, n)$,  $t=2,3,4$, are immediate.
%We compare the coefficients in the degree conditions of Conjecture~\ref{conj:mat} and Theorem~\ref{prop:cons}. First we claim that
To see \eqref{eq:Cor1}, it suffices to show that for $t\ge 1$,
\begin{align}
1 - f(t) %\binom t{ \lfloor t/2 \rfloor} \frac{ \lceil t/2\rceil^{\lceil t/2\rceil} (\lfloor t/2 \rfloor+1) ^{\lfloor t/2 \rfloor}}{(t+1)^t}
> 1 - \frac1{\sqrt{3t/2+1}}. \label{eq:2}
\end{align}
%It suffices to show that $f(d) < \left(1- \frac{1}{k} \right)^{d}$ for all $k\ge 4$ and $1< d \le k-1$.
%First, it is easy to see that $\left(1- \frac{1}{k} \right)^{d} \ge \left(1- \frac{1}{k} \right)^{k-1} > 1/e$.
When $t$ is odd, $\frac{ \lceil t/2\rceil^{\lceil t/2\rceil} (\lfloor t/2 \rfloor+1) ^{\lfloor t/2 \rfloor}}{(t+1)^t} = 1/2^t$; when $t$ is even, ${ \lceil t/2\rceil^{\lceil t/2\rceil} (\lfloor t/2 \rfloor+1) ^{\lfloor t/2 \rfloor}} < (\frac{t+1}2)^t$.
Thus, for all $t$, we have
\[
f(t)\le \binom t{ \lfloor t/2 \rfloor } \frac{1} {2^t},
\]
where a strict inequality holds for all even $t$.
Now we use the fact $\binom{2m}m \le 2^{2m}/\sqrt{3m+1}$, which holds for all integers $m\ge 1$.
Thus, for all even $t$, we have $f(t) \le 1/\sqrt{3t/2 +1}$; for all odd $t$,
\[
f(t) \le \binom t{ \lfloor t/2 \rfloor} \frac{1} {2^t} = \frac12 \binom {t+1}{ \lfloor t/2 \rfloor + 1} \frac{1} {2^t} \le \frac{1}{\sqrt{3(t+1)/2+1}} < \frac{1}{\sqrt{3t/2+1}}.
\]
Hence $f(t)\le 1/\sqrt{3t/2+1}$ for all $t\ge 1$. Moreover, by the computation above, regardless of the parity of $t$, the strict inequality always holds and thus \eqref{eq:2} is proved.

\smallskip
We next show that whenever $k\ge 4$ and $2\le t\le k-1$,
\[
1 - f(t) > \max \left\{ \frac12, 1- \left(1- \frac{1}{k} \right)^{t}  \right\}.
\]
This implies that Conjecture~\ref{conj:2} fails for $k\ge 4$, and Conjecture~\ref{conj:1} fails for $k\ge 4$ and $\min\{k-4, k/2\}\le d\le k-2$ (because $m_d(k, n)/\binom{n}{k-d} = \max \left\{ \frac12, 1- \left(1- \frac{1}{k} \right)^{k-d}  \right\} + o(1)$ in this case). It suffices to show that for $k\ge 4$ and $2\le t\le k-1$,
\[
f(t) < 1/2\, \text{ and }\, f(t) < \left(1- \frac{1}{k} \right)^{t}.
\]
The first inequality immediately follows from \eqref{eq:2} and $1/\sqrt{3t/2+1} \le 1/2$.
For the second inequality, note that
\[
f(t) < \frac{1}{\sqrt{3t/2+1}} < \frac 1e < \left(1- \frac{1}{k} \right)^{k-1} \le \left(1- \frac{1}{k} \right)^{t}
\]
for all $t\ge 5$.
For $t=2,3$ and all $k\ge 4$, one can verify $f(t) < (3/4)^t \le \left(1- \frac{1}{k} \right)^{t}$ easily. Also, for $t=4$ and all $k\ge 5$, we have $f(4) < (4/5)^4 \le (1- \frac{1}{k} )^{4}$.
%For $t=2,3,4$ and all $k\ge 4$, one can verify $f(t) < (1 - \frac1{\max\{t+1, 4\}})^t \le \left(1- \frac{1}{k} \right)^{t}$ easily.
\end{proof}

%\section{Concluding Remarks}

\bibliographystyle{amsplain}
\bibliography{Apr2015}

\end{document}